\documentclass{amsart}
\usepackage{amssymb}

\newcommand{\w}{\omega}
\newcommand{\IN}{\mathbb N}
\newcommand{\diam}{\mathrm{diam}}

\newcommand{\De}{\mathsf{I\hskip-1pt R}}
\newcommand{\Tau}{\mathcal T}
\newcommand{\Ra}{\Rightarrow}
\newcommand{\IR}{\mathsf{I\hskip-1pt R}}
\newcommand{\R}{\mathsf R}
\newcommand{\e}{\varepsilon}
\newcommand{\vH}{\vec{\mathcal H}}
\newcommand{\vTau}{\vec{\mathcal T}}
\newcommand{\Erdos}{\mbox{\sf{Erd\H os}}}
\newcommand{\Burr}{\mbox{\sf Burr}}
\newcommand{\E}{{\mathbb E}}
\newcommand{\dist}{\mathrm{dist}}

\newcommand{\ddiam}{\bar{\bar\chi}}
\newcommand{\A}{\mathcal A}

\newcommand{\beq}[1]{\begin{equation}\label{#1}}
\newcommand{\eeq}{\end{equation}}

\newtheorem{theorem}{Theorem}[section]
\newtheorem{proposition}[theorem]{Proposition}
\newtheorem{corollary}[theorem]{Corollary}
\newtheorem{problem}[theorem]{Problem}
\newtheorem{question}[theorem]{Question}
\newtheorem{lemma}[theorem]{Lemma}
\theoremstyle{definition}

\newtheorem{remark}[theorem]{Remark}

\title{Isometric copies of directed trees in orientations of graphs}
\author{Taras Banakh, Adam Idzik, Oleg Pikhurko, Igor Protasov, Krzysztof
Pszczo\l a}
\address{T.Banakh: Institute of Mathematics, Jan Kochanowski University in Kielce (Poland) and Ivan Franko University of Lviv (Ukraine)}
\email{t.o.banakh@gmail.com}
\address{A.Idzik:  Institute of Mathematics, Jan Kochanowski University in Kielce (Poland) and Institute of Computer Science, Polish Academy of Sciences, Warsaw (Poland)}
\email{adidzik@gmail.com}
\address{O.Pikhurko: Mathematics Institute and DIMAP, University of Warwick, Coventry, UK}
\thanks{O.~Pikhurko was supported by ERC
grant~306493 and EPSRC grant~EP/K012045/1.}
\email{O.Pikhurko@warwick.ac.uk}

\address{I.Protasov: Faculty of Cybernetics, Taras Shevchenko National University in Kyiv, Ukraine}
\email{i.v.protasov@gmail.com}

\address{K.Pszczo\l a: Instytut Matematyki i Kryptologii, Wojskowa Akademia Techniczna, Warsaw, 
Poland}
\email{pszczola@fr.pl}


\begin{document}
\begin{abstract}
The \emph{isometric Ramsey number} $\IR(\vec {\mathcal H})$ of a family $\vec{\mathcal H}$  of
digraphs is the smallest number of vertices in a graph $G$ such that any
orientation of the
edges of $G$ contains every member of $\vec H$ in the distance-preserving way.
We observe that for any finite family $\vec{\mathcal H}$ of finite acyclic graphs the  isometric Ramsey number $\IR(\vec {\mathcal H})$ is finite, and present upper bounds for $\IR(\vec {\mathcal H})$ in some special cases. For example, we show
that the
isometric Ramsey
number of the family of all oriented trees with $n$ vertices is at most
$n^{2n+o(n)}$.
\end{abstract}
\maketitle

\section{Introduction}

In this paper we consider the ``isometric'' version of the result of Cochand and Duchet \cite{CD}  who proved (generalizing a result of R\"odl \cite{Rodl}) that for every acyclic digraph $\vec H$ there exists a finite graph $G$ such that every orientation of $G$ contains an isomorphic copy of $\vec H$.

First we recall the necessary definitions from Graph Theory. A {\em graph}
is a pair $G=(V_G,E_G)$ consisting of a set $V_G$ of {\em vertices}
and a set $E_G$ of two-element subsets of $V_G$, called the {\em edges} of
$G$. By a {\em digraph} we will mean a pair $\vec G=(V_{\vec G},E_{\vec G})$
consisting of
a set $V_{\vec G}$ of vertices and a set $E_{\vec G}\subset V_{\vec G}\times
V_{\vec G}$ of {\em directed edges}, where neither loops $(x,x)$, nor pairs of
opposite arcs $(x,y)$ and $(y,x)$ are allowed.
 By an {\em orientation} of a graph
$G=(V_G,E_G)$ we understand a function $\vec{\cdot}:E_G\to V_G^2$ assigning to
each edge $e\in E_G$ an ordered pair $\vec e=(a,b)\in V_G^2$ such that
$e=\{a,b\}$. In this case the pair $\vec G=(V_G,\{\vec e\}_{e\in E_G})$ is a
digraph called an {\em orientation} of $G$.

A sequence $(v_0,\dots,v_n)$ of distinct vertices of a graph $G$ is called a
{\em path} in $G$ if for every positive $i\le n$ the unordered pair
$\{v_{i-1},v_i\}$ is an edge of $G$.
 The {\em length} of the path $(v_0,\dots,v_n)$ is $n$, that is, the number of
 edges. The
\emph{distance} $d_G(x,y)$ between two vertices $v,u$ of a graph $G$ is the smallest length of a path in $G$ connecting the vertices $v$ and $u$. If $u$ and $v$ cannot be connected by a path, then we write $d_G(x,y)=\infty$ and assume that $\infty>n$ for all $n\in\w$. A graph $G$ is called {\em connected} if any two vertices $u,v$ can be connected by a path in $G$. The distance in a digraph is taken with
respect to the underlying undirected graph.

A sequence $(v_0,\dots,v_n)$ of distinct vertices of a digraph $\vec G$ is called a {\em directed path} in $\vec G$ if for every positive $i\le n$ the ordered pair $(v_{i-1},v_i)$ is an edge of $G$. A \emph{directed cycle} is a sequence $(v_0,\dots,v_n)$ of distinct vertices with $(x_i,x_{i+1})$ being a directed edge for each residue $i$ modulo $n+1$.
A digraph $\vec G$ is {\em acyclic} if it contains no directed cycles. It
is well-known that each graph $G$ admits an acyclic orientation $\vec G$: take
any linear order $\le$ on the set $V_G$ of vertices and for any edge
$\{u,v\}\in E_G$ put $(u,v)\in E_{\vec G}$ if and only if $u<v$.

Following Rado's arrow notations, for a graph $G$ and a digraph $\vec H$ we
write $G\to \vec H$ if for every orientation $\vec G$ of $G$ there exists an injective
function $f:V_{\vec H}\to V_G$ such that an ordered pair $(u,v)$ of vertices of
$\vec H$ is
a directed edge in $\vec H$ if and only if $(f(u),f(v))$ is a directed edge in
$\vec G$. (Thus we require that $f$ induces an isomorphism of undirected graphs
and preserves all edge orientations.) If, moreover, $d_{\vec
H}(u,v)=d_G(f(u),f(v))$
for
every pair of vertices $u,v\in V_{\vec H}$, then we write $G\Ra \vec H$ and say
that
$f$ is an
\emph{isometric embedding} of $\vec H$ in $\vec G$. Since each graph $G$ admits
an
acyclic orientation, the arrow $G\to \vec H$ implies that the digraph $\vec H$
is acyclic.

Given a graph $G$ and a class $\vec{\mathcal H}$ of digraphs, we write $G\to
\vec{\mathcal H}$ (resp. $G\Ra \vec{\mathcal H}$) if  for every oriented graph
$\vec H\in\vec{\mathcal H}$ we have $G\to\vec H$ (resp. $G\Ra \vec H$). In this
case the family $\vec {\mathcal H}$ necessarily consists of acyclic digraphs.
For a natural number $n\in\IN$ by $\vec\Tau_n$ we denote the class of oriented
trees on $n$ vertices. By a {\em tree} we understand a connected graph without
cycles. For $n\in\IN$, the \emph{directed path} $\vec I_n$ is the digraph with
$V_{\vec I_n}=\{0,\dots,n-1\}$ and $E_{\vec I_n}=\{(i-1,i):0<i<n\}$.

For a class $\vec{\mathcal H}$ of digraphs let $\R(\vec{\mathcal H})$ (resp. $\IR(\vec{\mathcal H})$) be the smallest number of vertices of a graph $G$ such that $G\to\vec{\mathcal H}$ (resp. $G\Ra\vec{\mathcal H}$). If no graph $G$ with  $G\to\vec{\mathcal H}$ (resp. $G\Ra\vec{\mathcal H}$) exists, then we put $\R(\vec{\mathcal H})=\infty$ (resp. $\IR(\vec{\mathcal H})=\infty$). The number
$\R(\vec{\mathcal H})$ (resp. $\IR(\vec{\mathcal H})$) is called the ({\em isometric}) {\em Ramsey number} of the family $\vec{\mathcal H}$. If the family $\vH$ consists of a unique digraph $\vec H$, then we write $\R(\vec H)$ and $\IR(\vec H)$ instead of $\R(\{\vec H\})$ and $\IR(\{\vec H\})$, respectively.

By Theorem B of Cochand and Duchet \cite{CD}, for every finite acyclic digraph $\vec H$, the Ramsey number $\R(\vec H)$ is finite. This implies that for every finite family $\vec{\mathcal H}$ of finite acyclic digraphs the Ramsey number $\R(\vH)\le \sum_{\vec H\in\vH}\R(\vec H)$ is finite, too. In Section~\ref{s:Hoka} we shall apply a deep Ramsey result of Dellamonica and R\"odl \cite{DR} to prove that the isometric Ramsey number $\IR(\vH)$ is finite, too.

For the family $\vec\Tau_n$ of oriented trees on $n$ vertices Kohayakawa, \L
uczak and R\"odl \cite{KLR} proved that $\R(\vTau_n)=O(n^4\log n)$. In this
paper for every $n\in\IN$ we construct a graph $G_n$ with $<2^{2^{n-1}}$
vertices such that $G_n\Ra\vTau_n$, showing that $\IR(\vTau_n)<
2^{2^{n-1}}$. Using Bollob\'as' \cite{Bolobas78} bounds on the order of graphs
of large girth and large chromatic number, we shall improve the upper bounds
$\IR(\vec I_n)\le\IR(\vTau_n)< 2^{2^{n-1}}$ to $\IR(\vec I_n)=o(n^{2n})$ and
$\IR(\vTau_n)=o(n^{4n})$. In Theorem~\ref{t:Pikh} using random graphs we
improve the latter upper bound to $\IR(\vTau_n)\le (4e+o(1))^n(n^2\ln
n)^n=n^{2n+o(n)}$. The technique developed for the proof of
Theorem~\ref{t:Pikh} allows us to improve the upper bound $\R(\vTau_n)\le (2500
e^8+o(1))\,n^4\ln n$ obtained by Kohayakawa, \L uczak and R\"odl \cite{KLR} to
the upper bound $(K+o(1))\,n^4\ln n$,  where $K=\min_{x>1}\frac{16x^2}{1-x+x\ln
x}\approx 98.8249...$\ . In Section~\ref{s:chi} we search for long directed
paths in arbitrary orientations of graphs. In the final Section~\ref{s:infty}
we prove that every infinite graph $G$ admits an orientation containing no
directed path of infinite diameter in $G$. Some other results and problems
related to coloring and orientations in graphs can be found in \cite{Prot1}.

\section{The isometric Ramsey number for a finite acyclic digraph}\label{s:Hoka}

In this section we prove that each finite acyclic digraph $\vec H$ has finite isometric Ramsey number $\IR(\vec H)$. The idea of the proof of this result was suggested to the authors by Yoshiharu Kohayakawa.

\begin{theorem}\label{t:Koha} For any finite acyclic digraph $\vec H=(V,\vec E)$, the isometric Ramsey number $\IR(\vec H)$ is finite.
\end{theorem}

\begin{proof} Clearly, it is enough to prove the theorem when the graph $\vec H$ is connected. Fix any vertex $h$ of $H$ and consider the digraph $\vec \Gamma$ with $$V_{\vec \Gamma}:=V_{\vec \Gamma}\times\{0,1\}\mbox{ \ and \ }\vec E_{\vec\Gamma}:=\big\{\big((h,0),(h,1)\big)\big\}\cup\big\{\big((u,0),(v,0)\big),\big((v,1),(u,1)\big):(u,v)\in E_{\vec H}\big\}.$$
Observe that the digraph $\vec\Gamma$ is acyclic, connected and contains isometric copies of $\vec H$ and the graph $\vec H$ with the opposite orientation.
Being acyclic, the graph $\vec \Gamma$ admits a linear ordering $<$ of vertices such that $u<v$ for any directed edge $(u,v)\in \vec E_{\vec \Gamma}$.

By Theorem 1.8 of \cite{DR}, there exists a finite graph $G$ with a linear
ordering of vertices such that for any 2-coloring of its edges there exists a
monotone isometric embedding $f:V_{\vec \Gamma}\to V_G$ such that the set
$\big\{\{f(u),f(v)\}:(u,v)\in E_{\vec \Gamma}\big\}$ is monochrome. In this case we
shall say that the embedding $f$ is \emph{monochrome}. The monotonicity of $f$
means that $f$ preserves the order of vertices.

We claim that $G\Ra \vec H$. Given any orientation $\vec G$ of the graph $G$, color an edge $\{u,v\}\in E_G$ with $u<v$ in green if $(u,v)\in E_{\vec G}$ and in red if $(v,u)\in E_{\vec G}$. By the Ramsey property of $G$, there exists a monochrome monotone isometric embedding $f:V_{\vec \Gamma}\to V_G$. If the color of the monochromatic set $C=\big\{\{f(u),f(v)\}:(u,v)\in E_{\vec\Gamma}\big\}$ is green, then the map $g_0:V_{\vec H}\to V_G$, $g_0:v\mapsto f(v,0)$, is a required isometric isomorphic embedding of $\vec H$ into $\vec G$. If the color of $C$ is red, then the map $g_1:V_{\vec H}\to V_G$, $g_1:v\mapsto f(v,1)$, is an isometric isomorphic embedding of $\vec H$ into $\vec G$. In both cases we get $G\Ra\vec H$.
\end{proof}

\begin{corollary} Any finite family $\vec{\mathcal H}$ of finite acyclic digraphs has finite isometric Ramsey number $\IR(\vec{\mathcal H})$.\qed
\end{corollary}

\begin{corollary} For every $n\in\IN$ the family $\vTau_n$ of directed trees on $n$ vertices has finite isometric Ramsey number $\IR(\vTau_n)$.\qed
\end{corollary}

\begin{remark} The proof of \cite[Theorem 1.8]{DR} proceeds by a more general induction involving amalgamation and hypergraphs, and seems to give very bad bounds on the isometric Ramsey number $\IR(\vec\A_n)$ for the family $\vec\A_n$ of all acyclic digraphs on $n$ vertices. It would be interesting to get some reasonable upper bound on this function.
\end{remark}

\section{Simple bounds for the isometric Ramsey numbers $\IR(\vTau_n)$}

In this section we prove some simple upper bounds on the isometric Ramsey numbers $\IR(\vTau_n)$ and $\IR(\vec I_n)$. First we present a simple example of a graph witnessing that $\IR(\vTau_n)<2^{2^{n-1}}$.
The construction of this graph exploits  rectangular products of graphs.
By definition, the {\em rectangular product} $G\times H$ of two graphs $G,H$ is
a graph such that $V_{G\times H}=V_G\times V_H$  and an unordered pair
$\{(g,h),(g',h')\}\subset G\times H$ is an edge of $G\times H$ if and only if
either $\{g,g'\}\in E_G$ and $h=h'$ or $g=g'$ and $\{h,h'\}\in E_H$. It can be
shown that for any vertices $(g,h),(g',h')$ of $G\times H$ we get $$d_{G\times
H}\big((g,h),(g',h')\big)=d_G(g,g')+d_H(h,h').$$

For an (oriented) graph $G$ by $|G|$ we denote the cardinality of the set $V_G$ of vertices of $G$. For a cardinal number $m$ by $K_m$ we denote the complete graph on $m$ vertices.

\begin{lemma}\label{t1} Let $\vec {\Tau},\vec{\Tau}'$ be two families of finite oriented trees such that for every oriented tree $\vec T'\in\vec\Tau'$ there is an oriented subtree $\vec T\in\vec\Tau$ of $\vec T'$ such that $|\vec T|=|\vec T'|-1$. For any graph $G$ with $G\Ra\vec\Tau$ we get $G\times K_{|G|+1}\Ra \vec\Tau'$.
\end{lemma}

\begin{proof} Let $G'=G\times K_{|G|+1}$. To prove that $G'\Ra \vec\Tau'$, take any oriented tree $\vec T'\in\vec\Tau'$ and any orientation $\vec G'$ of the graph $G'$. By our assumption, for the tree $\vec T'$ there exists an oriented subtree $\vec T\in\vec\Tau$ of $\vec T'$ such that $|\vec T|=|\vec T'|-1$. Let $t'$ be the unique element of the set $V_{\vec T'}\setminus V_{\vec T}$ and $t\in V_{\vec T}$ be the unique vertex of $\vec T$ such that $(t',t)$ or $(t,t')$ is an edge of $\vec T'$.

For every vertex $u$ of the complete graph $K_{|G|+1}$, consider the  subgraph $ G'_u=G'\times\{u\}$ of $G'$ and its orientation $\vec G'_u$, inherited from the orientation $\vec G'$ of  $G'$. Since $G\Ra \vec\Tau$, there is an isometric embedding $f_u:\vec T\to \vec G'_u$. By the Pigeonhole Principle, there are two distinct vertices $u,w$ in $K_{|G|+1}$ such that $f_u(t)=(g,u)$ and $f_w(t)=(g,w)$ for some vertex $g$ of the graph $G$. Now look at the orientation of the edges $\{t,t'\}$ and $\{(g,u),(g,w)\}$ in the digraphs $\vec T'$ and $\vec G'$.

If either $(t,t')\in E_{\vec T'}$ and $\big((g,u),(g,w)\big)\in E_{\vec G'}$ or $(t',t)\in E_{\vec T'}$ and $\big((g,w),(g,u)\big)\in E_{\vec G'}$, then we define a map $f:\vec T'\to G'$ by $f(t')=(g,w)$ and $f|\vec T=f_u$ and observe that $f$ is an isometric embedding of $\vec T'$ into $\vec G'$.

If either $(t,t')\in E_{\vec T'}$ and $\big((g,w),(g,u)\big)\in E_{\vec G'}$ or $(t',t)\in E_{\vec T'}$ and $\big((g,u),(g,w)\big)\in E_{\vec G'}$, then we define a map $f:\vec T'\to G'$ by $f(t')=(g,u)$ and $f|\vec T=f_w$ and observe that $f$ is an isometric embedding of $\vec T'$ into $\vec G'$.
\end{proof}

\begin{corollary}\label{c1} If for some $n\in\IN$ a graph $G$ satisfies the isometric Ramsey relation $G\Ra\vec\Tau_n$, then $G\times K_{|G|+1}\Ra\vec\Tau_{n+1}$.\qed
\end{corollary}

\begin{theorem}\label{t2} For every $n\in\IN$ \ $\IR(\vec\Tau_{n+1})\le\De(\vec\Tau_n)(\De(\vec\Tau_n)+1)$ and $\De(\vec\Tau_n)< 2^{2^{n-1}}$.
\end{theorem}

\begin{proof} The inequality $\De(\vec\Tau_{n+1})\le \De(\vec\Tau^n)(\De(\vec\Tau^n)+1)$ follows from Corollary~\ref{c1}. Indeed, for every $n\in\w$ we can choose a graph $G$ with $|G|=\De(\vec \Tau^n)$ vertices and $G\Ra\vec\Tau_n$. By Corollary~\ref{c1}, the graph $G'=G\times K_{|G|+1}$ satisfies the relation $G'\Ra\vec\Tau_{n+1}$ and hence $$\De(\vec\Tau_{n+1})\le|G'|=|G|(|G|+1)=\De(\vec\Tau_n)(\De(\vec\Tau_n)+1).$$

It remains to prove that $\De(\vec\Tau_n)+1\le 2^{2^{n-1}}$ for $n\in\IN$. For $n=1$ we get the equality $\De(\vec\Tau_1)+1=1+1=2^{2^{0}}$. Assume that for some $n\in\IN$ we have proved that $\De(\vec\Tau_n)+1\le 2^{2^{n-1}}$. Then
$$\De(\vec\Tau_{n+1})+1\le\De(\vec\Tau_n)(\De(\vec\Tau_n)+1)+1\le (2^{2^{n-1}}-1)2^{2^{n-1}}+1=2^{2^n}-2^{2^{n-1}}+1\le 2^{2^{n}}.$$
\end{proof}

The upper bound $\IR(\vTau_n)< 2^{2^{n-1}}$ can be greatly improved using known upper bounds on the Erd\H os function $\Erdos(k,g)$, which assigns to any positive integer numbers $k,g$  the smallest cardinality $|G|$ of a graph $G$ with chromatic number $\chi(G)\ge k$ and girth $g(G)\ge g$. We recall that the  {\em girth} $g(G)$ of a graph is the smallest cardinality of a cycle in $G$. If $G$ contains no cycles, then we put $g(G)=\infty$. The {\em chromatic number} $\chi(G)$ of a graph $G$ is the smallest number $k\in\IN$ for which there exists a map $\chi:V_G\to\{1,\dots,k\}$ such that $\chi(x)\ne \chi(y)$ for any edge $\{x,y\}\in E_G$. The following  bounds for the Erd\H os function $\Erdos(k,g)$ were proved by Erd\H os \cite{Erdos}, Bollob\'as \cite{Bolobas78} and Spencer \cite{Spencer}, respectively.

\begin{proposition}
\begin{enumerate}
\item For any $k,g$ we get $\Erdos(k,g)\ge k^{(g-1)/2}$;
\item For any $k,g\ge 4$ we get $\Erdos(k,g)\le\lceil h^g\rceil$ where $h=6(k+1)\ln (k+1)$.
\item There exists a constant $C$ such that for any numbers $k,g\ge 3$ and
$m=\Erdos(k,g)$ we have the inequality $\sqrt[g-2]{m}\cdot \ln m<Ck$, which
implies that  $\Erdos(k,g)=o(k^{g-2})$ as $\max\{k,g\}\to\infty$.\qed
\end{enumerate}
\end{proposition}

Write $G\rightharpoonup \vH$ if for every orientation $\vec G$ of $G$ and every
$\vec H\in \vH$ there is an injective map $f:V_{\vec H}\to V_G$ such that for
every
directed edge
$(x,y)$ of $H$ the pair $(f(x),f(y))$ is a directed edge of $\vec G$. (Note that
we do not require that $f$ induces isomorphism, that is, $G$ can have extra
edges inside the set $f(V_{\vec H})$.) 
Another function related to $\IR(\vH)$ is Burr's function $\Burr(\vH)$
assigning to every family $\vH$ of oriented trees the smallest number $k$ such
that $G\rightharpoonup \vH$ for every graph $G$ with chromatic number
$\chi(G)\ge k$. If such number $k$ does not exist, then we put
$\Burr(\vH)=\infty$.
By the Gallai-Hasse-Roy-Vitaver Theorem \cite[Theorem~3.13]{Spar},
the chromatic number $\chi(G)$ of a finite
graph $G$ is equal to $\max\{n\in\IN:G\rightharpoonup \vec
I_n\}$.
 This equality implies that $\Burr(\vec I_n)=n$ for every $n\in\IN$. In
\cite{Burr} Burr considered the numbers $\Burr(\vTau_n)$ and proved that
$\Burr(\vTau_n)\le (n-1)^2$. This upper bound was improved to the upper bound
$\Burr(\vTau_n)\le \frac12n^2-\frac12n+1$ in \cite{AHSRT}. According to (still
unproved) Conjecture of Burr \cite{Burr}, the equality $\Burr(\vTau_n)=2n-2$
holds for all $n\ge 2$.

\begin{proposition}\label{Erdo-Bur} For any $n\in\IN$ and a subclass $\vH\subset \vTau_n$ we get the upper bound
$$\IR(\vH)\le\Erdos(\Burr(\vH),2n-2).$$
\end{proposition}

\begin{proof} Fix a graph $G$ of cardinality $|G|=\Erdos(\Burr(\vH),2n-2)$ with
chromatic number $\chi(G)\ge \Burr(\vH)$ and girth $g(G)\ge 2n-2$. Let us prove
that
$G\Ra\vH$. Take any  orientation $\vec G$ of $G$ and $\vec H\in\vH$.
Since $G\rightharpoonup\vH$, there is an orientation-preserving injection
$f:\vec
H\to\vec G$. Since $\vec H$ is a connected graph with at most $n$ vertices and
$g(G)\ge 2n-2$, the map $f$ is an isometric embedding. So,
$G\Ra\vH$.
\end{proof}

Combining Proposition~\ref{Erdo-Bur} with known upper bounds $\Burr(\vec I_n)=n$ and $\Burr(\vTau_n)\le\frac12n^2-\frac12n+1$ we get the following upper bounds for the isometric Ramsey numbers $\IR(\vec I_n)$ and $\IR(\vTau_n)$.

\begin{corollary}\label{c:erdo-bound} For every $n\in\IN$ we get the upper bounds
$$
\begin{aligned}
\IR(\vec I_n)&\le \Erdos(n,2n-2)=o(n^{2n-4})=o(n^{2n})\mbox{ and}\\
\IR(\vTau_n)&\le\Erdos\big(\tfrac12n^2-\tfrac12n+1,2n-2\big)=o\big((\tfrac12n^2-\tfrac12n+1)^{2n-4}\big)=o(n^{4n}).
\end{aligned}
$$
\mbox{ }\qed
\end{corollary}

In Theorem~\ref{t:Pikh} we shall improve the upper bound $o(n^{4n})$ for $\IR(\vTau_n)$ to the upper bound $n^{2n+o(n)}$.

\begin{remark} By Theorem 3 in \cite{KLR}, $\R(\vec I_n)\ge n^2/2$ for all $n\in\IN$. This yields the lower bound $$\tfrac12n^2\le \R(\vec I_n)\le\IR(\vec I_n)\le\IR(\vTau_n)$$
for the isometric Ramsey numbers $\IR(\vec I_n)$ and $\IR(\vTau_n)$.
\end{remark}

\begin{remark}
It can be shown that
$$
\begin{aligned}
&\IR(\vec I_1)=\De(\vec\Tau_1)=1=|K_1|,\\
&\IR(\vec I_2)=\De(\vec\Tau_2)=2=|K_2|,\\
&\IR(\vec I_3)=5=|C_5|,\;\De(\vec\Tau_3)=6=|K_2\times K_3|,\\
&\IR(\vec I_4)\le 30=|C_5\times K_6|,\;\De(\vec\Tau_4)\le 42=|K_2\times K_3\times K_{7}|.
\end{aligned}
$$
\end{remark}

\begin{question} What is the exact value of the isometric Ramsey numbers $\IR(\vec I_4)$ and $\IR(\vec\Tau_4)$? Are they distinct? \end{question}

\section{Isometric copies of directed trees in orientations of random graphs}\label{s:pikh}

In this section we shall apply the technique of random graphs and shall improve the upper bound $\IR(\vTau_n)=o(n^{4n})$ established in Corollary~\ref{c:erdo-bound} to the upper bound $\IR(\vTau_n)\le (4e+o(1))^n(n^2\ln n)^n=n^{2n+o(n)}$.

First we prove some technical lemmas. The first of them uses the idea of the proof of Theorem 1 in \cite{KLR}.

\begin{lemma}\label{l:Pikh} A graph $G=(V_G,E_G)$ satisfies $G\Ra\vTau_n$ for some $n\in\IN$ if there exist sequences $(w_k)_{k=1}^{n-1}$ and $(d_k)_{k=1}^{n-1}$ of positive real numbers such that for every  $2\le k<n$ the following conditions hold:
 \begin{enumerate}
 \item For every set $S=\{s_1,\dots,s_{k-1}\}\subset V_G$ of cardinality $k-1$ and every $v\in V_G\setminus S$, we have that
 $|Y|\le d_k$, where $Y$ consists of $y\in V_G\setminus(S\cup\{v\})$ such that $\{y,v\}\in E_G$ and $\dist_{G-v}(y,s_i)\le i$ for some $1\le i<k$.
 \item Every set $W\subset V_G$ of cardinality $|W|>w_k$ spans more than $(d_k+k-1)w_k$ edges in $G$.
   \item $\sum_{k=1}^{n-1}w_k<|V_G|$.
 \end{enumerate}
 \end{lemma}

\begin{proof} For a subset $U\subset V_G$ by $G[U]$ we denote the induced subgraph $G[U]=(U,E[U])$ of $G$, where $E[U]=\big\{\{u,v\}\in E_G:\{u,v\}\subset U\big\}$. Also, let us
write $(G,U)\Ra \vTau_k$, meaning that, for every $\vec T\in\vTau_k$,
every orientation $\vec G$ of $G$ contains a copy of $\vec T$ which lies inside $U$
and is an isometric subgraph of $G$.

We shall inductively prove that for every $1\le k\le n$ and every set $U\subset V_G$
of size $|U|>\sum_{i=1}^{k-1} w_i$, we have $(G,U)\Ra \vTau_k$.
 The base case $k=1$ is trivial. Suppose that this holds for some $k$.
Take any $U\subset V_G$ with $|U|>\sum_{i=1}^{k}w_i$. Take any orientation $\vec E(G[U])$ of $E(G[U])$ and any directed tree $\vec T\in \vTau_{k+1}$. Let $u$ be a pendant
vertex of $\vec T$. By symmetry, assume that $(v,u)$ is an arc in $\vec T$, that is, the arc in $\vec T$ goes from the unique neighbor $v$ of $u$ to $u$.

Let $W$ be the set of vertices in $U$ whose out-degree in $G[U]$ is at most $d_k+k-1$. We claim
that $|W|\le w_k$. Suppose not. Then $|W|>w_k$ and Item 2 guarantees that $W$ spans more than $(d_k+k-1)w_k$ edges in $G$, each edge contributing to out-degree
of some vertex in $W$. Thus $(d_k+k-1)|W|\ge |E[W]|>(d_k+k-1)w_k$, which is a desired contradiction showing that $|W|\le w_k$.

Thus $U'=U\setminus W$ has size $|U'|=|U|-|W|> \big(\sum_{i=1}^{k}w_i\big)-w_k=\sum_{i=1}^{k-1}w_i$. By inductive assumption, $U'$ has a $G$-isometric copy $\vec T'$ of the oriented tree $\vec T-u$. Let $\{s_1,\dots,s_{k-1}\}$ be an enumeration of the set $S:=V_{\vec T'}\setminus\{v\}\subset U'$ such that $\dist(s_i,v)\le i$ for every $i<k$.  Let $Y$ be defined as in Item~1 with respect to $v$ and $\{s_1,\dots,s_{k-1}\}$.
By Item 1, $|Y|\le d_k$. On the other hand, the neighbor
$v\in V_{\vec T'}\subset U\setminus W$ of $u$ must have out-degree in $U\setminus S$ greater than $d_k+k-1-|S|=d_k$.
Thus there is an out-neighbor of $v$
which is in $U\setminus(W\cup Y)$. Let $u$ be mapped to this vertex. Then $(v,u)\in \vec E(G[U])$ is
oriented from $v$ to $u$, as desired. Since $d_{G-v}(u,s_i)>i$ for each $i<k$, the addition of $u$ cannot violate the $G$-isometry property
(since all vertices of $\vec T-u$ are embedded into $S\cup\{v\}$). This gives the required embedding of $\vec T$ and finishes the proof.
\end{proof}

Our next elementary lemma yields an upper bound on the sum of a geometric progression.

\begin{lemma}\label{l:geom} For positive real numbers $a,c$ with $a>1+\frac1c$ we get $\frac{a^n-1}{a-1}<(1+c)a^{n-1}$ for every $n\in\IN$.
\end{lemma}

\begin{proof} The inequality is equivalent to $a^n-1<(1+c)a^{n-1}(a-1)=a^n-a^{n-1}+ca^{n-1}(a-1)$ and to $a^{n-1}-1<ca^{n-1}(a-1)$. The latter inequality follows from $a^{n-1}<ca^{n-1}(a-1)$, which is equivalent to $1<c(a-1)$.
\end{proof}

In the proof of Lemma~\ref{l:random} we shall use the following Chernoff-type
bounds;
for a proof see e.g.~\cite[\S A.1]{AS}.

\begin{lemma}[Chernoff bounds] Let $X_1,\dots,X_n$ be independent random variables taking values in $\{0,1\}$ and let $\mathbb E X$ be the expected value of their sum $X=\sum_{i=1}^nX_i$. Then $$
\mathbb P\big\{X\ge C\cdot\mathbb EX\big\}\le \big(\tfrac{e^{C-1}}{C^C}\big)^{\mathbb E X},\;\;\mathbb P\big\{X\ge(1+c)\mathbb EX\big\}\le e^{-\frac{c^2}{3}\mathbb E X}\mbox{ \ and \ \ }\mathbb P\big\{X\le(1-c)\mathbb EX\big\}\le e^{-\frac{c^2}2\mathbb E X}$$for every $C>1$ and $0<c<1$.\qed
\end{lemma}

\begin{lemma}\label{l:random} For positive integers $n,N$ the inequality $\IR(\vec\Tau_n)\le N$ holds if there exist real numbers  $c,p\in(0,1)$, $C\in(1,\infty)$ satisfying the following inequalities:
\begin{enumerate}
\item $c^2pN>3\ln(3N)$;
\item $(1-C+C\ln C)p(1+c)^n(pN)^{n-2}>(n-1)\ln N+\ln(1+c)+\ln(3)$;
\item $c^2C^2(1+c)^{2n}(pN)^{2n-4}>N\ln 2+\ln(3n)$;
\item $\frac{(n-1)(n-2)}{(1-c)p}+
\frac{2C}{(1-c)}(n-1)(1+c)^{n}(pN)^{n-2}<N$.
\end{enumerate}
\end{lemma}

\begin{proof} Assume that the numbers $n,N,p,c,C$ satisfy the assumptions of the lemma. Let $G=G(N,p)$ be a random graph on $N$ vertices in which an edge $\{u,v\}\subset V_G$ appears with probability $p$. We shall prove that with non-zero probability the random graph $G$ has $G\Ra\vec\Tau_n$.

Let $$\hbar:=(1+c)^{n}(pN)^{n-2}.$$
For every positive integer $k<n$ let $$d_k=Cp\hbar\mbox{ \ and \ }w_k= \dfrac{2(d_k+k-1)}{(1-c)p}.$$

Chernoff bound implies that any fixed vertex of $G$ has degree $\ge(1+c)p(N-1)$ with probability $<e^{-\frac{c^2}{3}p(N-1)}$. Consequently, with probability $P_1>1-Ne^{-\frac{c^2}3p(N-1)}$ all vertices of $G$ have degree $<(1+c)pN$. The condition (1) implies that $-\frac{c^2}3p(N-1)<-\ln(3N)$ and hence $$P_1>1-Ne^{-\frac{c^2}3p(N-1)}>1-Ne^{-\ln(3N)}=\tfrac23.$$

For every $k<n$, take any pairwise distinct points $v,s_1,\dots,s_{k-1}\in V_G$. If
the maximum degree of $G$ is at most $(1+c)pN$, then for every $i<k$ the ball $B(s_i,i)=\{x\in V_G:\dist_G(x,s_i)\le i\}$ has cardinality $$|B(s_i,i)|\le \sum_{j=0}^i\big((1+c)pN\big)^j=\frac{\big((1+c)pN\big)^{i+1}-1}{(1+c)pN-1}<
(1+c)\big((1+c)pN\big)^{i}.$$
The latter strict inequality can be derived from Lemma~\ref{l:geom} and the inequality $cpN\ge c^2pN>3\ln(3N)\ge 3$.

By above, the set $X$ of  vertices of $G-v$ at distance at most $i<k$ in $G-v$ from some $s_i$ has size at most $(1+c)\sum_{i=1}^{k-1}\big((1+c)pN\big)^{i}=(1+c)\frac{((1+c)pN)^k-1}{(1+c)pN-1}<
(1+c)^{k+1}(pN)^{k-1}\le \hbar$.

Consider the set $Y$ of neighbors of $v$ that fall into the set $X$.
 The definition of $X$ does not depend on the edges incident to $v$, so conditioned on $X$ (of size at most $\hbar$) the size of $Y$ is dominated
by $Y'\sim Bin(\hbar,p)$. Chernoff bound shows that the probability that $Y'$ is at least $Cp\hbar=C\E Y'$ is
at most $\big(\frac{e^{C-1}}{C^C}\big)^{p\hbar}$. Since the number of possible choices of $v,s_1,\dots,s_{k-1}$ is equal to $\frac{N!}{(N-k)!}\le N^k$, with probability
$$P_2\ge 1-\sum_{k=1}^{n-1}N^k\big(\tfrac{e^{C-1}}{C^C}\big)^{p\hbar}=1-
\big(\tfrac{e^{C-1}}{C^C}\big)^{p\hbar}\frac{N^n-1}{N-1}>
1-(1+c)N^{n-1}\big(\tfrac{e^{C-1}}{C^C}\big)^{p\hbar}$$ the condition (1) of Lemma~\ref{l:Pikh} is satisfied or we have a vertex of degree $\ge (1+c)pN$.
We claim that $P_2>\frac23$. It suffices to prove that
$$\ln(1+c)+(n-1)\ln N+p\hbar(C-1-C\ln C)<-\ln(3).$$ But this follows from condition (2).
\smallskip

Next, we prove that with probability $>\frac23$ the condition (2) of Lemma~\ref{l:Pikh} holds. Take any positive $k<n$ and put $\bar w_k=\min\{m\in\IN:w_k<m\}$. For any fixed set $W\subset V_G$ of cardinality $|W|=\bar w_k$, the number of edges it spans is $Bin({\bar w_k\choose 2},p)$. By Chernoff bound, the probability that it is less than $(1-c)p{\bar w_k\choose 2}$ is less that $e^{-\frac12c^2p{\bar w_k\choose 2}}$. The probability $P_{3,k}$ that some set $W\subset V_G$ of cardinality $|W|=\bar w_k$ spans less than $(1-c)p{\bar w_k\choose 2}$ edges is $P_{3,k}<{N\choose\bar w_k}e^{-\frac12c^2p{\bar w_k\choose 2}}<2^Ne^{-\frac14c^2p\bar w_k(\bar w_k+1)}$. We claim that $P_{3,k}<\frac1{3n}$ which will follow as soon as we show that
$N\ln 2-\frac14c^2p\bar w_k(\bar w_k+1))<-\ln(3n)$. For this it suffices to check that $\frac14c^2p\bar w_k(\bar w_k+1)>N\ln 2+\ln(3n)$.

This follows from the chain of the inequalities
$$
\begin{aligned}
\tfrac14c^2\bar w_k(\bar w_k+1)>\tfrac14c^2w_k^2>c^2C^2\hbar^2=
c^2C^2(1+c)^{2n}(pN)^{2n-4}>N\ln 2+\ln(3n),
\end{aligned}
$$
the last inequality postulated in (3).
Therefore, $P_{3,k}<\frac1{3n}$ and the probability $P_3$ that for every $k<n$ every set $W\subset V[G]$ of cardinality $|W|>w_k$ spans at least$$(1-c)p{\bar w_k\choose 2}>(1-c)pw_k(w_k+1)/2=(d_k+k-1)(w_k+1)>(d_k+k-1)w_k$$ edges is $>1-\sum_{k=1}^{n-1}P_{3,k}>1-\frac{n-1}{3n}>\frac23$. So, with probability $>\frac23$ the condition (2) of Lemma~\ref{l:Pikh} holds.

Since $(1-P_1)+(1-P_2)+(1-P_3)<1$, there is a non-zero probability that the random graph $G=G(N,p)$ satisfies the conditions (1) and (2) of Lemma~\ref{l:Pikh}.

It remains to show that the condition (3) of Lemma~\ref{l:Pikh} holds, too. For this observe that
$$
\begin{aligned}
\sum_{k=1}^{n-1}w_k&=\sum_{k=1}^{n-1}\frac{2(Cp\hbar+k-1)}{(1-c)p}=
\frac{2}{(1-c)p}\sum_{k=1}^{n-1}(k-1)+
\frac{2C}{1-c}(n-1)\hbar=
\\
&= \frac{(n-1)(n-2)}{(1-c)p}+
\frac{2C}{1-c}(n-1)(1+c)^{n}(pN)^{n-2}<N.
\end{aligned}
$$The last inequality follows from the condition (4) of the Lemma.
\smallskip

Now it is legal to apply Lemma~\ref{l:Pikh} and conclude that $G\Ra\vec\Tau_n$ and hence $\IR(\vec\Tau_n)\le|G|=N$.
\end{proof}

Now we are able to prove the promised upper bound $\IR(\vTau_n)\le (4e+o(1))^n(n^2\ln n)^{n}=n^{2n+o(n)}$.

\begin{theorem}\label{t:Pikh} For every $\e\in(0,1)$ there is $n_\e\in\IN$ such that $\IR(\vTau_n)\le \big(4e(1+\e)\,n^2\ln n\big)^n$ for all $n\ge n_\e$.
\end{theorem}

\begin{proof} Choose any positive  $\delta,c\in(0,1)$ such that $$(1+\delta)(1+c)<1+\e\mbox{ \ \ and \ \  }4(1+\delta)\,\frac{1-c}{2+c}>2+\delta.$$
For every $n\in\IN$ let $N$ be the smallest integer number, which is greater than
$$\tfrac{(2+c)e^n}{1-c}(n-1)(1+c)^n\big(4(1+\delta)\,n^2\ln n\big)^{n-2}$$ and let $$p:=\frac{4(1+\delta)\,n^2\ln n}N.$$
So, $N>\frac{(2+c)e^n}{1-c}(n-1)(1+c)^n(pN)^{n-2}\ge N-1$.
 It is easy to see that $$N=o\big((4e(1+\e)\,n^2\ln n)^n\big)$$ and for $C=e^n$ the conditions (1),(3),(4) of Lemma~\ref{l:random} hold for all sufficiently large $n$. To verify the condition (2), observe that
 $$
\begin{aligned}
&(1-C+C\ln C)p(1+c)^n(pN)^{n-2}\ge(1-e^n+e^n\ln e^n)p\frac{(N-1)(1-c)}{(2+c)e^n(n-1)}=\\
&\frac{1+e^n(n-1)}{e^n(n-1)}\frac{1-c}{2+c}\frac{N-1}{N}\,pN=
\Big(1+\frac1{e^n(n-1)}\Big)\frac{N-1}{N}\frac{1-c}{2+c}\,4(1+\delta)n^2\ln n>\\
&>\Big(1+\frac1{e^n(n-1)}\Big)\frac{N-1}{N}(2+\delta)\,n^2\ln n=(2+\delta+o(1))\,n^2\ln n.
\end{aligned}
$$
On the other hand, $(n-1)\ln N+\ln(1+c)+\ln 3=(2+o(1))\,n^2\ln n$. So, the condition (2) holds for large $n$. Applying Lemma~\ref{l:random}, we conclude that $$\IR(\vec\Tau_n)\le N\le (4e(1+\e)\,n^2\ln n)^n$$ for all sufficiently large $n$.
\end{proof}

By Corollary~\ref{c:erdo-bound} and Theorem~\ref{t:Pikh}, $\IR(\vec I_n)=o(n^{2n})$ and $\IR(\vTau_n)\le n^{2n+o(n)}$.

\begin{question} What is the growth rate of the sequence $\De(\vec\Tau_n)$?  Is $\De(\vec\Tau_n)=n^{o(n)}$? 
\end{question}

The technique developed for the proof of Theorem~\ref{t:Pikh} allows us to
improve the upper bound $$\R(\vTau_n)\le\big(4(5e^2)^4+o(1)\big)\,n^4\ln n,$$
obtained by Kohayakawa, \L uczak and R\"odl in (the proof of) Theorem 1 of
\cite{KLR}, and replace the constant $4(5e^2)^4=2500 e^8\approx 7452395.96...$
by the a much smaller constant $K\approx 98.82...$~. 

\begin{theorem}\label{t:KLR} Let $K:=\min_{x>1}\frac{16x^2}{1-x+x\ln x}\approx 98.8249...$ For any positive $\e>0$ there exists $n_\e\in\IN$ such that $\R(\vTau_n)<(K+\e)\,n^4\ln n$ for all $n\ge n_\e$. Consequently, $\R(\vTau_n)<99\,n^4\ln n$ for all sufficiently large $n$.
\end{theorem}

\begin{proof} We indicate which changes should be made in the proof of Theorem~\ref{t:Pikh} to obtain Theorem~\ref{t:KLR}.
\smallskip

In the condition (1) of Lemma~\ref{l:Pikh} the inequality $d_{G-v}(y,s_i)\le i$ should be replaced by $d_{G-v}(y,s_i)\le 1$.

In the proof of Lemma~\ref{l:random} the constant $\hbar$ should be redefined as $\hbar:=(1+c)(n-2)pN$ and the conditions (1)--(4) of Lemma~\ref{l:random} should be changed to the conditions:
\begin{itemize}
\item[(1')] $c^2pN>3\ln(3N)$;
\item[(2')] $(1-C+C\ln C)(1+c)(n-2)p^2N>(n-1)\ln N+\ln(1+c)+\ln(3)$;
\item[(3')] $(cC(1+c)(n-2)pN)^2>N\ln 2+\ln(3n)$;
\item[(4')] $\frac{n(n-1)}{(1-c)p}+\frac{2C(1+c)}{(1-c)}(n-1)(n-2)pN<N$.
\end{itemize}

Now we are able to prove Theorem~\ref{t:KLR}. Let $C\approx 4.92155...$  be the unique real number in $(1,\infty)$ such that
$$\frac{16C^2}{1-C+C\ln C}=K:=\min_{x>1}\frac{16x^2}{1-x+x\ln x}\approx 98.8249...\footnote{The approximate values of $C$ and $K$ were found by the online WolframAlpha computational knowledge engine at www.wolframalpha.com}$$
Given any $\e>0$, choose real numbers $\delta,c\in (0,1)$  such that $K\delta<\e$ and
$$
4(1+\delta)\frac{(1-c)^2}{(1+c)^3}>4+\delta.
$$

For every $n\in\IN$ let $p:=\frac{1-c}{2C(1+c)^2n^2}$ and let $N$ be the smallest integer, which is greater than $K(1+\delta)n^4\ln n$. It is easy to see that $N=o\big((K+\e)\,n^4\ln n\big)$ and the conditions (1'), (3') and (4') are satisfied for all sufficiently large $n$. To see that (2') holds, observe that
$$
\begin{aligned}
&(1-C+C\ln C)(1+c)(n-2)p^2N\ge \frac{(1-C+C\ln C)(1+c)(1-c)^2}{(2C(1+c)^2n^2)^2}(n-2)K(1+\delta)\,n^4\ln n=\\
&=\frac{1-C+C\ln C}{C^2}\frac{(1-c)^2}{4(1+c)^3}(n-2)K(1+\delta)\ln n=(1+\delta)K\frac{16}{K}\frac{(1-c)^2}{4(1+c)^3}(n-2)\ln n>\\
&>(4+\delta)(n-2)\ln n=(4+\delta+o(1))\,n\ln n.
\end{aligned}
$$
On the other hand,
$$(n-1)\ln N+\ln(1+c)+\ln 3\le (n-1)\ln(1+K(1+\delta)\,n^4\ln n)+\ln(1+c)+\ln 3=(4+o(1))\,n\ln n,$$
so for large $n$ the condition $(2')$ is satisfied, too.

Applying the modified version of Lemma~\ref{l:random}, we get $$\R(\vTau_n)\le N\le (K+\e)\,n^4\ln n$$ for all sufficiently large numbers $n$.
\end{proof}

\section{Long directed paths in orientations of a graph}\label{s:chi}

By the Gallai-Hasse-Roy-Vitaver Theorem \cite[Theorem~3.13]{Spar}, each finite
graph $G$ has chromatic number $$\chi(G)=\max\{n\in\IN:G\rightharpoonup \vec
I_n\},$$where the symbol $G\rightharpoonup \vec
I_n$ means that each orientaion of $G$ contains a simple directed path of length $n$.  Having in mind this characterization, for every graph $G$ consider the
numbers
$$\ddiam_I(G)=\sup\{n\in\IN:G\Ra\vec I_n\},\;\ddiam_T(G)=\sup\{n\in\IN:G\Ra\vec\Tau_n\},\;$$
and observe that $\ddiam_T(G)\le\ddiam_I(G)\le\chi(G)$ and $$\ddiam_I(G)\le\sup\{\diam(G')+1:\mbox{$G'$ is a connected component of $G$}\}.$$
Observe that $\IR(\vec I_n)$ (resp. $\IR(\vTau_n)$) is equal to the smallest cardinality $|G|$ of a graph $G$ with $\ddiam_I(G)\ge n$ (resp. $\ddiam_T(G)\ge n$). So, the characteristics $\ddiam_I$ and $\ddiam_T$ determine the isometric Ramsey numbers $\IR(\vec I_n)$ and $\IR(\vTau_n)$.
\smallskip

We shall show that a graph $G$ has $\ddiam_I(G)\le 2$ if and only if $G$ is a comparability graph.
We recall that a graph $G$ is called a {\em comparability graph} if $G$ admits
a \emph{transitive} orientation $\vec G$ (that is, for any directed edges $(x,y)$ and $(y,z)$ of $\vec G$ the
pair $(x,z)$ is a directed edge of $\vec G$); equivalently, the set $V_G$ of vertices of $G$ admits
a partial order such that a pair $\{u,v\}$ of distinct vertices of $G$ is an
edge of $G$ if and only if $u$ and $v$ are comparable in the partial order. By
the results of Ghouila-Houri and of Gilmore and Hoffman (see
\cite[Theorem~6.1.1]{BS}), comparability graphs can be characterized as graphs
$G$ whose every cycle of odd length has a triangular chord (more precisely, for
every $(2n+3)$-cycle on $(v_0,\dots, v_{2n+2})$ with $n\ge
1$, there is
a residue $i$ modulo $2n+3$ such that
$\{v_{i},v_{i+2}\}\in E_G$). More information on comparability graphs can be
found in Chapter 6 of the survey \cite{BS}.

\begin{proposition} A graph $G$ has $\ddiam_I(G)\le 2$ if and only if $G$ is a comparability graph.
\end{proposition}

\begin{proof} If $G$ is comparability graph, then $G$ has a transitive orientation $\vec G$. It follows that for any directed path $(v_0,v_1,v_2)$ in $\vec G$ the pair $(v_0,v_2)$ is an edge of $\vec G$ and hence $d_G(v_0,v_2)\le 1$. This means that $G\not\Ra \vec I_3$ and hence $\ddiam_I(G)\le 2$.

If $G$ is not a comparability graph, then $G$ contains an odd cycle $C$ without a triangular chord. It is easy to see that any orientation $\vec C$ of the cycle $C$ contains a directed path $(v_0,v_1,v_2)$. Since $C$ has no triangular chords, $d_G(v_0,v_2)=2$, which means that $\{v_0,v_1,v_2\}$ is an isometric copy of $\vec I_3$ in $\vec C$ and in $G$. Therefore, $\ddiam_I(G)\ge 3$.
\end{proof}

\begin{problem} Characterize graphs $G$ with $\ddiam_I(G)\le 3$ \textup{(} $\ddiam_I(G)\le n$ for $n\ge 4$ \textup{)}.
\end{problem}

\begin{problem} Characterize graphs $G$ with $\ddiam_T(G)\le 2$ \textup{(} $\ddiam_T(G)\le n$ for $n\ge 3$ \textup{)}.
\end{problem}

\begin{remark} Any cycle $C$ of odd length $n\ge 5$ satisfies $\ddiam_I(C)=3$ and $\ddiam_T(C)=2$.
\end{remark}

Now we prove a weak 3-space property for the number $\ddiam_I(G)$. By a {\em weak homomorphism} $f:G\to H$ of graphs $G,H$ we understand a function $f:V_G\to V_H$ such that for every edge $\{u,v\}$ of $G$ we have either $f(u)=f(v)$ or  $\{f(u),f(v)\}$ is an edge of $H$. For a weak homomorphism $f:G\to H$ and vertex $y$ of $H$ the preimage $f^{-1}(y)$ is a graph with the set of edges $\big\{\{u,v\}\in E_G:f(u)=y=f(v)\big\}$.

\begin{proposition} If $f:G\to H$ is a weak homomorphism of finite graphs, then
$$\ddiam_I(G)\le \max\big\{\sum_{y\in F}\ddiam_I(f^{-1}(y)):F\subset V_H,\;|F|\le\chi(H)\big\}.$$
\end{proposition}

\begin{proof} By definition of the chromatic number $\chi(H)$, there exists a coloring $c:V_H\to \{1,\dots,\chi(H)\}$  of the graph $H$ such that for every edge $\{u,v\}$ of $G$ the colors $c(u)$ and $c(v)$ are distinct. For every $y\in H$ choose an orientation $\vec G_y$ of the graph $G_y=f^{-1}(y)$ such that $\vec G_y\not\Ra \vec I_k$ for $k=\ddiam_I(G_y)+1$. Let $\vec G$ be the orientation of the graph $G$ such that for an edge $\{u,v\}$ of $G$ the ordered pair $(u,v)$ is an edge of $\vec G$ if and only if either $c(f(u))<c(f(v))$ or $(u,v)$ is an edge of $\vec G_y$ for some $y\in H$. 

We claim that the digraph $\vec G$ contains no isometric copy of the graph $\vec I_{m+1}$, where
$$m= \max\Big\{\sum_{y\in F}\ddiam_I(G_y):F\subset V_H,\;|F|\le\chi(H)\Big\}.$$
Suppose on the contrary that $\vec G$ contains a directed path $(v_0,\dots,v_m)$ such that $d_G(v_0,v_m)=m$. It follows that $\big(c(f(v_0)),\dots,c(f(v_n))\big)$ is a non-decreasing sequence of numbers in the interval $\{1,\dots,\chi(H)\}$. Consequently,  for every number $i$ in the set $C=\{c(f(v_0)),\dots,c(f(v_n))\}$ the set $J_i=\{j\in \{0,\dots,n\}:c(f(v_j))=i\}$ coincides with some subinterval $[a_i,b_i]$ of  $\{0,\dots,n\}$ and the set $\{f(v_j):j\in[a_i,b_i]\}$ is a singleton $\{y_i\}$ for some vertex $y_i\in H$. It follows that  $(v_{a_i},\dots,v_{b_i})$ is a directed path isometric to $\vec I_{|[a_i,b_i]|}$ in the graph $G_{y_i}$ and hence $|[a_i-b_i]|\le\ddiam_I(G_{y_i})$. The choice of the orientation $\vec G$ guarantees that the set $F=\{y_i:i\in C\}$ has cardinality $|F|=|C|\le\chi(H)$. Then
$$
m+1=|[0,m]|=\sum_{i\in C}|[a_i,b_i]|\le\sum_{i\in C}\ddiam_I(G_{y_i})=\sum_{y\in F}\ddiam_I(G_y)\le m,
$$
which is a desired contradiction.
\end{proof}

\section{Infinite directed paths in orientations of graphs}\label{s:infty}

Now we discuss the problem of existence of infinite directed paths in orientations of graphs. Consider the infinite digraphs $\vec I_\w$ and $\vec I_{-\w}$ with $V_{\vec I_\w}=\w=V_{\vec I_{-\w}}$, $E_{\vec I_\w}=\{(i,i+1):i\in\w\}$, and
$E_{\vec I_{-\w}}=\{(i+1,i):i\in\w\}$.

First, observe that Theorem~\ref{t2} implies the following:

\begin{corollary} There exists a countable graph $G$ such that $G\Ra\vec I_n$ for every $n\in\IN$.\qed
\end{corollary}

On the other hand, we shall prove that each graph $G$ admits an orientation
containing no isometric copy of the digraphs $\vec I_\w$ or $\vec I_{-\w}$ and,
more generally, no directed paths of infinite diameter in $G$. (For a subset
$A\subset V_G$ of a graph $G$ its {\em diameter} is defined as
$\diam(A)=\sup\{d_G(u,v):u,v\in A\}\in\w\cup\{\infty\}$.)

A sequence $(v_n)_{n\in\w}\in V_G^\w$ of distinct vertices of a graph $G$ is
called an {\em $\w$-path} in $G$ if for every $n\in\w$ the pair
$\{v_n,v_{n+1}\}$ is an edge of $G$. An $\w$-path $(v_n)_{n\in\w}$ in a graph
$G$ is called {\em $\overrightarrow\w$-directed} (resp.
$\overleftarrow{\w}$-{\em directed}) in an orientation $\vec G$ of $G$ if for
every $n\in\w$ the pair $(v_n,v_{n+1})$ (resp. $(v_{n+1},v_n)$) is a directed
edge of $\vec G$.
An $\w$-path in $G$ is called {\em directed} in an orientation $\vec G$ of $G$ if it is either $\overrightarrow\w$-directed or $\overleftarrow\w$-directed.

The Ramsey Theorem implies that every orientation of the complete countable graph $K_\w$ contains $\vec I_\w$ or~$\vec I_{-\w}$. On the other hand, we have the following result.

\begin{theorem}\label{t3} Every graph $G$ has an orientation $\vec G$ containing no directed $\w$-paths of infinite diameter in $G$. This implies that $G\not\Ra\vec I_\w$ and $G\not\Ra\vec I_{-\w}$.
\end{theorem}

\begin{proof} Without loss of generality, the graph $G$ is connected. Fix any vertex $o$ in $G$ and for every vertex $v$ of $G$ let $\|v\|$ be the smallest length of a path linking the vertices $v$ and $o$. Choose an orientation $\vec G$ of $G$ such that for any edge $\{u,v\}$ in $G$ with $\|v\|=\|u\|+1$ the pair $(u,v)$ is an edge of $\vec G$ if $\|u\|$ is even and $(v,u)$ is an edge of $\vec G$ if $\|u\|$ is odd.

We claim that the orientation $\vec G$ contains no directed $\w$-paths of infinite diameter. To derive a contradiction, assume that $(v_n)_{n\in\w}$ is a directed $\w$-path of infinite diameter. Fix any even number $n\in\w$ such that $\|v_0\|<n$. Since the $\w$-path $(v_n)_{n\in\w}$ has infinite diameter, there exists a number $k\in\w$ such that $\|v_k\|\ge n$. We can assume that $k$ is the smallest number with this property. Taking into account that $\big|\|v_n\|-\|v_{n+1}\|\big|\le 1$ for all $n\in\w$, we conclude that $\|v_k\|=n>\|v_0\|$ and $\|v_{k-1}\|=n-1$, and hence $(v_{k-1},v_k)$ is an edge of $\vec G$. Let also $m$ be the smallest number such that $\|v_m\|\ge n+1$. For this number we get $\|v_m\|=n+1$, $\|v_{m-1}\|=n$ and hence $(v_{m},v_{m-1})$ is a directed edge $\vec G$. Since both pairs $(v_{k-1},v_k)$ and $(v_m,v_{m-1})$ are directed edges of the oriented graph $\vec G$, the $\w$-path $(v_n)_{n\in\w}$ is not directed in $\vec G$.
Since the graphs $\vec I_\w$ and $\vec I_{-\w}$ have infinite diameters, the digraph $\vec G$ does not contain isometric copies of $\vec I_\w$ or $\vec I_{-\w}$.
\end{proof}

\begin{remark} Theorem~\ref{t3} implies that every locally finite graph $G$
admits an orientation containing no directed  $\w$-paths.
\end{remark}

\section{Acknowledgments}

The authors express their sincere thanks to Yoshiharu Kohayakawa for suggesting the proof of Theorem~\ref{t:Koha} (which resolves a problem posed in an earlier version of this paper) and to Tomasz \L uczak for valuable discussions related to random graphs.


\end{document}